%% file: main.tex
\documentclass[11pt]{amsart}
\usepackage{hyperref,graphicx,color,cleveref,enumerate}
\usepackage[margin=1in]{geometry}

\newtheorem{lemma}{Lemma}[section]
\newtheorem{theorem}[lemma]{Theorem}

\newtheorem{proposition}[lemma]{Proposition}
\theoremstyle{definition}
\newtheorem{definition}[lemma]{Definition}

\newtheorem{example}[lemma]{Example}

\newtheorem{mainthm}{Theorem}

\crefname{mainthm}{Theorem}{Theorems}

\theoremstyle{remark}
\newtheorem{remark}[lemma]{Remark}

\usepackage[nodisplayskipstretch]{setspace}
\setdisplayskipstretch{1.5}
\allowdisplaybreaks
\numberwithin{table}{section}

\newcommand{\ZZ}{\mathbb Z}

\newcommand{\QQ}{\mathbb Q}

\newcommand{\CC}{\mathbb C}

\newcommand{\cO}{\mathcal O}

\newcommand{\cT}{\mathcal T}

\newcommand{\cF}{\mathcal F}
\newcommand{\hQ}{\widehat{\mathbb Q}}

\DeclareMathOperator{\tr}{tr}
\DeclareMathOperator{\SL}{SL}

\newcommand{\bmat}[4]{\begin{bmatrix}#1 & #2 \\ #3 & #4\end{bmatrix}}
\newcommand{\bmatcol}[2]{\begin{bmatrix}#1\\ #2\end{bmatrix}}

\begin{document}
\title[Two-bridge knot epimorphisms and intersections]{Two-bridge knot epimorphisms and intersections in the character variety}

\author[Bailly-Hall]{Isidora Bailly-Hall}
\address{University of Minnesota - Twin Cities}
\email{ibaillyh@umn.edu}

\author[Chu]{Michelle Chu}
\address{University of Minnesota - Twin Cities}
\email{mchu@umn.edu}

\author[Dovgodko]{Karina Dovgodko}
\address{Brown University}
\email{karina\_dovgodko@brown.edu}

\author[Ganguly]{Akash Ganguly}
\address{University of Minnesota - Twin Cities}
\email{aganguly@umn.edu}

\author[Kang]{Jiachen Kang}
\address{Harvard University}
\email{jkang@math.harvard.edu}

\author[Sun]{Jishi Sun}
\address{University of Michigan}
\email{joshsun@umich.edu}

\begin{abstract}
We study the $\SL_2(\CC)$ character varieties of hyperbolic two-bridge knot groups admitting epimorphisms onto the trefoil knot group. Ohtsuki, Riley, and Sakuma give a systematic construction of an infinite family of such knots. The character varieties of these knot groups contain a component of irreducible representations arising from the epimorphism, which is distinct from the canonical component. We prove that every affine point of intersection between this component and any other component is algebraic non-integral, and therefore detects an essential surface in the knot complement. We also determine the boundary slope of each detected surface explicitly in terms of the continued fraction expansion of the two-bridge knot.
\end{abstract}


\maketitle

\input{Introduction}

\input{2Bridge}

\input{Farey}

\input{Vanishing}

\input{Slopes}

\section*{Acknowledgements}
This work is a result of the University of Minnesota School of Mathematics REU in Algebra and Combinatorics during Summer 2023, funded by NSF RTG grant DMS-1745638. The group thanks TAs Tori Braun and Lilly Webster for their guidance during the REU. They also thank Zeus Dantas e Moura for his help with diagrams. MC was partially supported by NSF grants DMS-2243188 and DMS-2441034.

\section*{AI Disclosure}
All mathematical content was developed without the aid of artificial intelligence. The authors used Claude by Anthropic at the final stage of preparation for proofreading, grammatical edits, and checking notation consistency across the text.

\bibliographystyle{alpha}
\bibliography{bib.bib}

\end{document}

%% file: Introduction.tex
\section{Introduction}
\label{sec:intro}

The $\SL_2(\CC)$ character variety has long been an important tool in the study of 3-manifolds.
In \cite{cullershalen}, Culler and Shalen give a general approach based on Bass-Serre theory to construct essential surfaces in the complement of hyperbolic knots using their $\SL_2(\CC)$ character varieties. These essential surfaces arise from non-trivial actions of the knot group on $\SL_2$-trees, and are said to be \emph{detected} by $\SL_2$-trees. Culler and Shalen also describe one method for constructing such trees by using ideal points in $\SL_2(\CC)$ character varieties. Another method of constructing $\SL_2$-trees is via algebraic non-integral (ANI) representations of hyperbolic knot groups. \cite{SchanuelZhang} compared these two methods and proved that any essential surface detected by ideal points is also detected by ANI-representations.

The $\SL_2(\CC)$ character variety of a hyperbolic knot group contains multiple irreducible components, including a \emph{canonical component}, which contains the character of a holonomy representation of the hyperbolic structure.
In \cite{michelle}, the second author studied a family of hyperbolic two-bridge knots whose character varieties are known to contain two distinct irreducible components corresponding to irreducible characters, and proved that the affine points at the intersection of these two components always detect a Seifert surface. In this paper, we extend this pattern to a different family of hyperbolic two-bridge knots: those whose knot groups admit epimorphisms onto the trefoil knot group.

The $\SL_2(\CC)$ character varieties of these knot groups contain an irreducible component corresponding to the character variety of the trefoil knot group, and this component intersects a canonical component at finitely many affine points.

In \cite{ORS}, Ohtsuki, Riley, and Sakuma (ORS) gave a systematic method to construct an infinite family of two-bridge knots whose knot groups admit epimorphisms onto a given two-bridge knot. It is an open question whether all epimorphisms between two-bridge knots arise from their construction. We focus on those two-bridge knot groups in the ORS construction that admit epimorphisms onto the trefoil knot group.
The trefoil knot is the simplest non-trivial two-bridge knot and hence the simplest instance of the ORS construction, making this family a natural starting point for studying the relationship between epimorphisms and essential surfaces.

\begin{mainthm}
    \label{thm:A}
Let $K$ be any hyperbolic two-bridge knot arising from the Ohtsuki-Riley-Sakuma construction, whose knot group $\Gamma$ admits an epimorphism onto the trefoil knot group. In the $\SL_2(\CC)$-character variety of $\Gamma$, any affine point of intersection between the component of irreducible representations that factor through the epimorphism onto the trefoil knot group and any other component of irreducible representations is algebraic non-integral.
\end{mainthm}

Two-bridge knot complements are known to be \emph{small}, in the sense that they do not contain any closed essential surfaces. Any essential surface detected by an ANI representation must have non-empty boundary, and is therefore characterized by its \emph{boundary slope}, the slope at which it meets the boundary torus.
In \cite{hatcherthurston}, Hatcher and Thurston classify all essential surfaces in two-bridge knot complements, using continued fraction expansions corresponding to knot diagrams in \emph{normal form}.

The continued fraction expansions for the knots we consider in \Cref{thm:A} have the form $[3\sigma_1,2\zeta_1,\dots,3\sigma_n,2\zeta_n,3\sigma_{n+1}] $
where $\sigma_i,\zeta_j\in \ZZ_{\ne 0}$ and $\sum_{i=1}^{n+1}\sigma_i\equiv 1\bmod 2$ (see \Cref{prop:ORS} and \Cref{lem:epsumodd} for details).

\begin{mainthm}
\label{thm:B}
Let $K$ and $\Gamma$ be as in \Cref{thm:A}. In the $\SL_2(\CC)$-character variety of $\Gamma$, any affine point of intersection between the component corresponding to the epimorphism to the trefoil knot group and any other component of irreducible representations detects an embedded essential surface in the complement of $K$ having boundary slope $6\sum_{i=1}^{n+1} \sigma_i$.
\end{mainthm}

This paper is organized as follows.
In \Cref{sec:2bridge} we provide the necessary background on two-bridge knots, their character varieties, and the ORS construction.
In \Cref{sec:Farey}, we introduce the Farey recursion framework of \cite{Chesebro,farey-recursion}, which gives an explicit recursive description of the character varieties of two-bridge knots, and collect the properties needed for the remainder of the paper.
In \Cref{sec:vanishing}, we use this framework to analyze the polynomials whose roots correspond to intersection points between the component arising from the epimorphism onto the trefoil knot group and other irreducible components. The section ends with the proof of \Cref{thm:A}.
In \Cref{sec:slope}, we determine the boundary slope detected at each such intersection point and prove \Cref{thm:B}.

%% file: 2Bridge.tex
\section{Two-bridge knots and their character varieties}
\label{sec:2bridge}

In \cite{Schubert}, Schubert showed that every two-bridge link can be associated to a rational number $\alpha$ by a diagram constructed from its continued fraction expansion $[a_1,a_2,\dots,a_m]$ where $(-1)^{i+1}a_i$ denotes the number of half-twists in the $i$th twist region as in \Cref{fig:knot}. Writing $\alpha$ as a reduced fraction $q/p$, the associated link is a knot whenever $p$ is odd, and is otherwise a 2-component link. The reduced fraction is not unique to the link. In the case of knots, we may choose $q/p$ with $0<q<p$ and $q,p$ both odd. The corresponding knot is a torus knot whenever $q=1$ and is otherwise hyperbolic. For example, $1/3$ corresponds to the trefoil knot.

\begin{figure}[ht]
\includegraphics[height=2.5in]{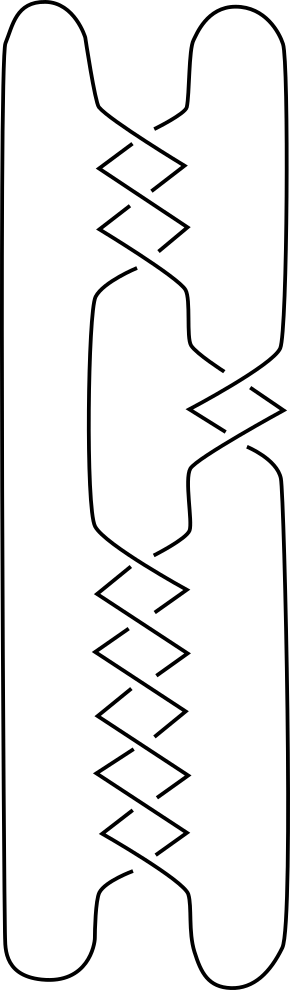}
\caption{Diagram of the two-bridge knot with continued fraction expansion $[3,2,6]$.}\label{fig:knot}
\end{figure}

Let $M_\alpha$ denote the $\mathbb{S}^3$-complement of the two-bridge link corresponding to $\alpha=q/p$ and $\Gamma_\alpha$ denote the fundamental group.
Following Schubert's description, a presentation for the fundamental group is given by
\begin{equation} \label{knotgrouppres}
\begin{gathered}
    \Gamma_{\alpha}=\langle a,b \mid wa=bw \rangle
    \\ \text{where }w=a^{\epsilon_1}b^{\epsilon_2}\dots a^{\epsilon_{p-2}}b^{\epsilon_{p-1}}
    \text{ and } \epsilon_i=(-1)^{\lfloor iq/p \rfloor}
\end{gathered}
\end{equation}
The generators $a,b$ represent meridians in $M_{q/p}$ and are conjugate in the knot group. We take $\mu=a$ to be the preferred meridian and $\lambda=w^*wa^\epsilon$ the corresponding longitude, where $w^*$ denotes $w$ written backwards and $\epsilon=-2\sum\epsilon_i$ chosen so $\lambda$ has exponent sum zero.

\subsection{Character varieties}
We refer the reader to \cite{shalen} for more on this material.
Let $\Gamma$ be a finitely generated group and let $X(\Gamma)$ denote the $\SL_2(\CC)$-character variety of $\Gamma$, that is the set
\[ X(\Gamma) = \{\chi_\rho \mid \rho:\Gamma\rightarrow\SL_2(\CC) \} \]
where the character $\chi_\rho$ of a representation $\rho$ is defined by $\chi_\rho(\gamma)=\tr(\rho(\gamma))$. The character variety $X(\Gamma)$ has the structure of an affine algebraic set.
The characters of reducible representations form an algebraic subset of $X(\Gamma)$ which we will call $X_{red}(\Gamma)$. The closure of the set of characters of irreducible representations will be denoted $X_{irr}(\Gamma)$. Any irreducible component which contains the character of a discrete and faithful representation is contained in $X_{irr}(\Gamma)$ and is denoted $X_0(\Gamma)$ and called a \emph{canonical component}.

In this paper, we are concerned with character varieties of two-bridge knots $\Gamma_\alpha$. In this case, $X(\Gamma_\alpha)$ is an algebraic curve in $\CC^2$ which is realized as the vanishing set of a polynomial $\psi_\alpha(x,y)\in\ZZ[x,y]$ in the variables
\[ x=\chi_\rho(a) \text{ and } y=\chi_\rho (ab^{-1}). \]
We will refer to $\psi_\alpha(x,y)$ as the \emph{defining polynomial} for $X(\Gamma_\alpha)$ and take it to be square-free.

\begin{example}\label{rem:trefoil}
Recall that $\Gamma_{1/3}$ is the trefoil knot group. Its character variety $X(\Gamma_{1/3})$ is cut out by the polynomial $(y-2)(x^2-y-1)$ and has two irreducible components corresponding to the two irreducible factors. Here the factor $(y-2)$ corresponds to the component $X_{red}(\Gamma_{1/3})$ of characters of reducible representations.
\end{example}

For general two-bridge knot groups, the component of characters of reducible representations will always correspond to a factor of $(y-2)$.

\subsection{Knot group epimorphisms}\label{subsec:epi}
If a knot group $\Gamma$ admits an epimorphism onto another knot group $\Gamma'$, then any representation of $\Gamma'$ induces a representation of $\Gamma$ which factors through the epimorphism.
Knot group epimorphisms guarantee the existence of non-canonical irreducible components. Indeed, Theorem 2.3(3) of \cite{hoste-shanahan} states that every component of $X(\Gamma')$ appears as a component of $X(\Gamma)$.

Ohtsuki, Riley, and Sakuma describe in \cite{ORS} how to systematically generate all pairs of two-bridge knot groups admitting epimorphisms. The following proposition is a special case of \cite[Proposition 5.1]{ORS}.

\begin{proposition}[Ohtsuki-Riley-Sakuma]\label{prop:ORS}
Let $K_\alpha$ be a two-bridge link with associated continued fraction expansion $\alpha$ of the form
\[ \alpha =[3\sigma_1,2\zeta_1,\dots,3\sigma_n,2\zeta_n,3\sigma_{n+1}] \]
where $\sigma_i,\zeta_j\in \ZZ_{\ne 0}$.
Then its link group $\Gamma_\alpha$ admits an epimorphism onto the trefoil knot group.
\end{proposition}

The form of the continued fraction expansion in \Cref{prop:ORS} differs from that stated in \cite[Proposition 5.1]{ORS} by a change of equivalent continued fraction expansion. We omit the straightforward but lengthy computation.

\begin{remark}
    The condition for $K_\alpha$ as in \Cref{prop:ORS} to be a knot rather than a 2-component link is that $\alpha=q/p$ where $p$ is odd. We will show in \Cref{lem:epsumodd} that this happens exactly when $\sum_{i=1}^{n+1}\sigma_i$ is odd.
\end{remark}

The main object of study of this paper is hyperbolic two-bridge knot groups $\Gamma_\alpha$ that have epimorphisms onto $\Gamma_{1/3}$. The character variety of such knot groups will have at least three components: $X_{red}(\Gamma_\alpha)$ cut out by $(y-2)$, a canonical component $X_0(\Gamma_\alpha)$, and a component $X_{tre}(\Gamma_\alpha)$ cut out by $(x^2-y-1)$ coming from the epimorphism onto the trefoil knot group.

\begin{definition}\label{def:phi}
Let $\alpha$ have continued fraction expansion as in \Cref{prop:ORS} and $\psi_\alpha(x,y)$ be the defining polynomial for $X(\Gamma_\alpha)$.
Define the \emph{vanishing polynomial}
\[ \phi_\alpha(x)= \frac{ \psi_\alpha(x,x^2-1)}{x^2-3}.  \]
\end{definition}

Since ${x^2-3}$ defines the component corresponding to the trefoil knot group, $\phi_\alpha(x)\in\ZZ[x]$.
Any point in the intersection of $X_{tre}(\Gamma_\alpha)$ with another component of $X(\Gamma_\alpha)$ of characters of irreducible representations has form $(x_0,x_0^2-1)$ where $x_0$ is a root of $\phi_\alpha$.
The roots of $\phi_\alpha$ are therefore precisely the $x$-coordinates of the intersection points we study in \Cref{thm:A} and \Cref{thm:B}.

\subsection{Detecting essential surfaces}
A character $\chi_\rho$ of representation $\rho:\Gamma\rightarrow\SL_2(\CC)$ with image contained in $\SL_2(F)$ for a number field $F$ is called \emph{algebraic non-integral}, or \emph{ANI} for short, if there exists a discrete valuation $\nu:F\rightarrow\ZZ$ and an element $\gamma\in\Gamma$ such that $\nu(\chi_\rho(\gamma))<0$. Such a valuation exists whenever $\chi_\rho(\gamma)$ is not an algebraic integer.
Whenever $\Gamma$ is the fundamental group of a 3-manifold, such an ANI character is said to ``detect'' an essential surface.
As in \cite{cullershalen}, an ANI character induces a non-trivial action on the Bruhat-Tits tree for $\SL_2$ with respect to the valuation $\nu$ and this action induces a splitting of the 3-manifold along an essential surface.

Two-bridge knot complements are small in the sense that they do not contain closed essential surfaces (see \cite{hatcherthurston}). Any essential surface will therefore meet the boundary torus along simple closed curves of the same slope. The set of slopes realized by essential surfaces in a two-bridge knot complement $M_\alpha$ are known as \emph{boundary slopes}. The following is a restatement of \cite[Corollary 3]{SchanuelZhang} in the case of two-bridge knots.

\begin{theorem}[Schanuel-Zhang]\label{boundary_slope}
Let $M_\alpha$ be a two-bridge knot and $\chi_\rho$ an ANI character of $\Gamma_\alpha$ with respect to a discrete valuation $\nu$. Then there exists a unique slope $\delta$ in the boundary of $M_\alpha$ such that $\nu(\chi_\rho(\delta))>0$, and an essential surface $S$ in $M_\alpha$ with non-empty boundary of slope $\delta$.
\end{theorem}

%% file: Farey.tex
\section{Farey recursions and characters of irreducible representations}
\label{sec:Farey}

In this section we introduce the Farey recursive framework developed by Chesebro and his collaborators in \cite{Chesebro,farey-recursion} which provides an explicit recursive description of the character varieties of two-bridge knots.
We collect here the properties of this framework that are needed in \Cref{sec:vanishing,sec:slope}.

\subsection{Farey pairs}
\label{subsec:farey}
If $q/p$ and $s/r$ are reduced fractions in $\hQ=\QQ\cup\{\infty\}$ such that $qr-ps=\pm1$, then $\{q/p,s/r\}$ is called a \emph{Farey pair}. We make the convention to write $1=1/1$, $0=0/1$, $\infty=1/0$, and $q/p$ with $p\geq0$.

Given a Farey pair $\{q/p,s/r\}$, their \emph{Farey sum} is defined to be
\[ \frac{q}{p}\oplus\frac{s}{r}=\frac{q+s}{p+r},
\]
which can be checked to be already a reduced fraction. One may similarly define their negative Farey sum to be
\[ \frac{q}{p}\oplus^{-1}\frac{s}{r}=\frac{q-s}{p-r},
\]
which can also be checked to be already a reduced fraction possibly with negative denominator.

A quick calculation shows that if $\{\alpha,\beta\}$ is a Farey pair, then the two triples $\{\alpha,\beta,\alpha\oplus\beta\}$ and $\{\alpha,\beta,\alpha\oplus^{-1}\beta\}$ form a \emph{Farey triple}, meaning a collection of three elements of $\hQ$ with the property that any two of them form a Farey pair.
We may therefore take repeated (positive and negative) Farey sums: define for a Farey pair $\{q/p,s/r\}$ and integer $k$,
\[\frac{q}{p}\oplus^k\frac{s}{r}= \frac{q+ks}{p+kr}.
\]

We use the notation in \cite{farey-recursion} to give a matrix form of continued fractions.
Consider $\alpha\in\hQ$ with continued fraction expansion $[a_1,\dots,a_m]$ of length $m$. For $1\leq j\leq m$, let $\alpha_j$ be the element of $\hQ$ with continued fraction expansion $[a_1,\dots,a_j]$. By convention, we choose to write $\alpha_j=q_j/p_j$ where $p_j\geq 0$. Define $\eta_j\in\{1,-1\}$ so that
\[ \begin{bmatrix} q_j  \\ p_j \end{bmatrix} = \eta_j
\begin{bmatrix} 0 & 1 \\ 1 & a_1 \end{bmatrix}
\begin{bmatrix} 0 & 1 \\ 1 & a_2 \end{bmatrix} \cdots
\begin{bmatrix} 0 & 1 \\ 1 & a_j \end{bmatrix}
\begin{bmatrix} 0  \\ 1 \end{bmatrix}.
\]

\begin{lemma}\label{lem:farey_truncs}
Suppose $\alpha\in\hQ$ has continued fraction expansion $[a_1,\dots,a_m]$. Let $\alpha_{-1}=1/0$ and $\alpha_0=0/1$. Then for all $0\leq j\leq m$, the pair $\{\alpha_{j-1},\alpha_j\}$ is a Farey pair. Furthermore,
\[ \alpha_j=\alpha_{j-2} \oplus^{\eta_{j-2}\eta_ja_j} \alpha_{j-1} .
\]
\end{lemma}

\begin{proof}
When $j=0$, $\{\alpha_{-1},\alpha_0\}=\{1/0,0/1\}$ is clearly a Farey pair. Continuing with induction on $j$, assume $\{\alpha_{j-2},\alpha_{j-1}\}$ is a Farey pair. Note that
\begin{align*} & \frac{1}{\eta_{j-2}}  \bmatcol{q_{j-2}}{p_{j-2}} + a_j \left(\frac{1}{\eta_{j-1}}  \bmatcol{q_{j-1}}{p_{j-1}}\right) \\
& = \bmat{0}{1}{1}{a_1}\cdots \bmat{0}{1}{1}{a_{j-2}}\bmatcol{0}{1}+a_j \left( \bmat{0}{1}{1}{a_1}...\bmat{0}{1}{1}{a_{j-1}}\bmatcol{0}{1}\right)\\
 & = \left( \bmat{0}{1}{1}{a_1}...\bmat{0}{1}{1}{a_{j-2}} \right) \left( \bmatcol{0}{1}+\bmat{0}{1}{1}{a_{j-1}}\bmatcol{0}{a_j} \right)\\
 & = \left( \bmat{0}{1}{1}{a_1}...\bmat{0}{1}{1}{a_{j-2}} \right) \left( \bmat{0}{1}{1}{a_{j-1}}\bmat{0}{1}{1}{a_j}\bmatcol{0}{1} \right)\\
  & = \frac{1}{\eta_{j}} \bmatcol{q_{j}}{p_{j}} .
\end{align*}
When $\eta_{j-2}=\eta_{j-1}$, the first equality is equal to $\pm\bmatcol{q_{j-2}+a_j q_{j-1}}{p_{j-2}+a_j p_{j-1}}$, which corresponds to the Farey sum $\alpha_{j-2}\oplus^{a_j} \alpha_{j-1}$.
When $\eta_{j-2}\neq\eta_{j-1}$, the first equality is equal to $\pm\bmatcol{q_{j-2}-a_j q_{j-1}}{p_{j-2}-a_j p_{j-1}}$, which corresponds to the fraction $\alpha_{j-2}\oplus^{-a_j} \alpha_{j-1}$. Therefore $\alpha_j=\alpha_{j-2} \oplus^{\eta_{j-2}\eta_ja_j} \alpha_{j-1}$ and it follows that $p_{j-1}q_{j}-p_{j}q_{j-1}=p_{j-1}q_{j-2}-p_{j-2}q_{j-1}$, so $\{\alpha_{j-1},\alpha_j\}$ is a Farey pair.
\end{proof}

\subsection{Farey recursive functions}

Let $R$ be a commutative ring. A function $\cF:\hQ\to R$ is called a \emph{Farey recursive function} (with determinant $1$ as in \cite[\S5.1]{Chesebro}) if for any Farey pair $(\alpha,\beta)\in \hQ^2$, we have
\begin{equation} \label{eq:frf 2}
	\cF(\beta\oplus^2\alpha)=-\cF(\beta)+\cF(\alpha)\mathcal{F}(\beta\oplus\alpha).
\end{equation}
It follows that whenever $\cF$ is a Farey recursive function, we also have
\begin{equation} \label{eq:frf-2}
	\cF(\beta\oplus^{-2}\alpha)=-\cF(\beta)+\cF(\alpha)\mathcal{F}(\beta\oplus^{-1}\alpha).
\end{equation}

Let $\cT:\hQ\rightarrow\ZZ[w,z]$ be the unique Farey recursive function satisfying
\[ \cT(0)=w \text{, } \cT(1/0)=0 \text{, and } \cT(1)=z,
\]
(see \cite[Theorem 4.2]{farey-recursion}) and define $\cT_0:\hQ\rightarrow \ZZ[w,z]$ by
\begin{equation} \label{eq:T0def}
 \cT_0(\alpha)=\frac{\cT(\alpha)}{f(\alpha)} \text{, with }
f(q/p) = \begin{cases}
w &\text{if $q$ is even}\\
z &\text{if $p$ and $q$ are odd}\\
wz &\text{if $p$ is even, $q$ is odd}.
\end{cases}
\end{equation}

The following lemma is a restatement of \cite[Lemma 5.5]{Chesebro}.

\begin{lemma}\label{lem:T0}
For all $\alpha\in \hQ$, we have $\cT_0(\alpha)\in \ZZ[w^2,z^2]$.
\end{lemma}

\subsection{Generating character varieties of two-bridge knots}
Chesebro is able to systematically describe the character varieties of 2-bridge knots in terms of the Farey recursive functions above.

\begin{lemma}\label{lem:psi}
Let $\alpha\in \QQ$ correspond to a two-bridge knot group $\Gamma_\alpha$. Then $\chi=(x_0,y_0)\in \mathbb C^2$ is the character of an irreducible representation if and only if $(y-2)(2+y-x^2)\neq 0$ and $\chi$ satisfies $\cT_0(\alpha)=0$ after replacing $w^2$ by $2+y-x^2$ and $z^2$ by $2-y$.
\end{lemma}

\begin{proof}
Our choice of variables $x,y$ for $X(\Gamma_\alpha)$ is different from that of \cite{Chesebro}. Instead, our $x$ matches the $x$ in \cite{MPvL}, and our $y$ matches the $r$ in \cite{MPvL}.
To correctly translate the change of variable, we refer to \cite[\S 7.2]{Chesebro} to get the substitutions $ x^2-2 = 2-w^2-z^2$ and $y=2-z^2$. This yields
\[ w^2=2+y-x^2 \text{ and } z^2=2-y . \]
The statement now follows from \cite[Theorem 7.2]{Chesebro}.
\end{proof}

%% file: Vanishing.tex
\section{Vanishing polynomials for intersection points}
\label{sec:vanishing}

In this section we use the general setup from \Cref{subsec:farey} to give a description of the vanishing polynomial $\phi_\alpha$ defined in \Cref{def:phi}. The section concludes with a proof of \Cref{thm:A}.

\subsection{Conditions on the continued fraction}
\label{subsec:conditionsonalpha}

Recall that $\alpha_j$ is obtained by concatenating the continued fraction for $\alpha$ after the $j$-th entry.

\begin{lemma}\label{lem:cons}
Consider $\alpha =[3\sigma_1,2\zeta_1,\dots,3\sigma_n,2\zeta_n,3\sigma_{n+1}] $ with $\sigma_i,\zeta_j\in \ZZ_{\ne 0}$.
If $j$ is odd, then \[\cT(\alpha_j)(1,z)=\cT_0(\alpha_j)(1,z)=0.\]
If $j$ is even, then the numerator of $\alpha_j$ is even, and 
\[\cT(\alpha_j)(1,z)=\cT_0(\alpha_j)(1,z)=(-1)^{\sum_{i=1}^{j/2}\zeta_i}.\]
\end{lemma}

\begin{proof} 
When $j$ is odd, the continued fraction expansion for $\alpha$ has form as in \Cref{prop:ORS}, so there is an epimorphism $\Gamma_{\alpha_j}\to \Gamma_{1/3}$. By \cite[Corollary 7.6]{Chesebro}, $\cT(1/3)=z(w^2-1)$ divides $\cT(\alpha_j)$. Setting $w=1$ yields $\cT(1/3)(1,z)=0$ so also $\cT(\alpha_j)(1,z)=0$ and $\cT_0(\alpha_j)(1,z)=0$.

Note first that $\alpha_2=(0\oplus^{2\zeta_1-2}\alpha_1)\oplus^2\alpha_1$ so then
\[ \cT(\alpha_2)=-\cT(0\oplus^{2\zeta_1-2}\alpha_1)+\cT(\alpha_1)\cT( (0\oplus^{2\zeta_1-2}\alpha_1)\oplus\alpha_1). \]
But the second term is 0 when $w=1$ because $\cT(\alpha_1)(1,z)=0$.

We now consider induction when $j$ is even. Note first that $\alpha_2=(0\oplus^{2\zeta_1-2}\alpha_1)\oplus^2\alpha_1$, so we use $\alpha_0=0$ as a stand-in base case. Clearly $0=0/1$ has even numerator and $\cT(0)(1,z)=1$.
By \Cref{lem:farey_truncs}, $\alpha_{j+2}=\alpha_j\oplus^{\pm 2\zeta_{1+j/2}}\alpha_{j+1}$. If the numerator of $\alpha_j$ is even, then the numerator of $\alpha_{j+2}$ is also even.
Using the Farey recursive function definition, we have that
\[ \cT(\alpha_{j+2})=-\cT(\alpha_j\oplus^{\pm |2\zeta_{1+j/2}-2|}\alpha_{j+1})+\cT(\alpha_{j+1})\cT(\alpha_j\oplus^{\pm |2\zeta_{1+j/2}-1|}\alpha_{j+1}).\]
However, we can ignore the second term since $j+1$ is odd so by above, $\cT(\alpha_{j+1})(1,z)=0$.
Now induction on $\zeta_1$ shows 
\[\cT(\alpha_{j+2})(1,z)=(-1)^{\zeta_{1+j/2}}\cT(\alpha_j)(1,z)=(-1)^{\sum_{i=1}^{1+j/2}\zeta_i}.\]
Finally, by the definition in \Cref{eq:T0def} with an even numerator, 
\[ \cT_0(\alpha_{j+2})=\cT(\alpha_{j+2})/w \]
so setting $w=1$ completes the proof. 
\end{proof}

We can now prove the condition that guarantees $\alpha$ corresponds to a knot and not a 2-component link.

\begin{lemma}\label{lem:epsumodd}
The continued fraction expansion \[\alpha=[3\sigma_1,2\zeta_1,\dots,3\sigma_n,2\zeta_n,3\sigma_{n+1}]\] with $\sigma_i,\zeta_j\in \ZZ_{\ne 0}$ has odd denominator if and only if $\sum_{i=1}^{n+1}\sigma_i$ is odd.
\end{lemma}

\begin{proof}
Note first that when $n=0$, $[3\sigma_1]=\frac{1}{3\sigma_1}$ has odd denominator whenever $\sigma_1$ is odd. We proceed by induction on $n$ and assume $[3\sigma_1,2c_1,\dots, 3\sigma_{n}]$ has odd denominator if and only if $\sum_{i=1}^n\sigma_i$ is odd. By \Cref{lem:farey_truncs},
\[
\alpha=[3\sigma_1,2c_1,\dots, 3\sigma_{n}]\oplus^{\pm 3\sigma_{n+1}}[3\sigma_1,2c_1,\dots, 2c_n].
\]
By \Cref{lem:cons}, the continued fraction expansion $[3\sigma_1,2c_1,\dots, 2c_n]$ has even numerator, so its denominator is odd. Therefore the denominator of $\alpha$ will have the same parity as the denominator of $[3\sigma_1,2c_1,\dots, 3\sigma_{n}]$ plus $3\sigma_{n+1}$, which is odd if and only if $\sum_{i=1}^{n+1}\sigma_i$ is odd.
\end{proof}

For the remainder of the paper, we fix
\[\alpha = \alpha_{2n+1}=[3\sigma_1,2\zeta_1,\dots,3\sigma_n,2\zeta_n,3\sigma_{n+1}]=[a_1,\dots,a_{2n+1}] \]
 with $\sigma_i,\zeta_j\in \ZZ_{\ne 0}$ and $\sum_{i=1}^{n+1}\sigma_i$ as in \Cref{prop:ORS} .

\subsection{Description of the polynomials}

\begin{lemma} \label{lem:Z[z^2]}
    Let $\gamma$ be a continued fraction expansion of form \[ \gamma =[3\sigma_1,2\zeta_1,\dots,3\sigma_k].\]
    The quantity
    \[ \frac{\cT(\gamma)}{z(w^2-1)} \]
    is in $\ZZ[w^2,z^2]$.
    Moreover, under the substitutions $w^2=1$ and $z^2=x^2-3$, this polynomial agrees with the vanishing polynomial $\phi_\gamma(x)$ from \Cref{def:phi}.
\end{lemma}

\begin{proof}
    Note first that $\gamma$ has ORS form as in \Cref{prop:ORS}, so $X_{tre}(\Gamma_\gamma)$ is an irreducible component of $X(\Gamma_\gamma)$. On $X_{tre}(\Gamma_\gamma)$ we have $y=x^2-1$, so the substitution in \Cref{lem:psi} gives $w^2=2+y-x^2=1$. Therefore $(w^2-1)$ divides $\cT_0(\gamma)$.
    Since $\gamma$ defines a rational number with odd numerator and odd denominator, we have $\cT_0(\gamma)=\cT(\gamma)/z$ by definition. So $z(w^2-1)$ divides $\cT(\gamma)$. The fact that the result lies in $\ZZ[w^2,z^2]$ follows from \Cref{lem:T0}. The identification with $\phi_\gamma$ follows by setting $z^2=x^2-3$, which corresponds to substitution $y=x^2-1$ in \Cref{def:phi}.
\end{proof}

\begin{lemma} \label{lem:Tmu}
    For $\alpha$ as in \Cref{prop:ORS} with $m=2n$, we have 
    \begin{enumerate}
        \item $ \cT(\alpha_{m-1})(1,z)=0 $
        \item $ \cT(\alpha_{m-1}\oplus^{\pm1} \alpha_m)(1,z)=(-1)^{\tau}z $ for some $\tau\in\{0,1\}$ 
        \item for all $\ell\geq 1$,
        $\cT(\alpha_{m-1}\oplus^{\pm (3\ell-2)}\alpha_m)(1,z) = (-1)^{\tau+(\ell-1)\left( 1+\sum_{i=1}^n \zeta_i \right)}z $
        and 
        $\cT(\alpha_{m-1}\oplus^{\pm (3\ell-3)}\alpha_m)(1,z) = 0 $
    \end{enumerate}
\end{lemma}

\begin{proof}
    (1) follows directly from \Cref{lem:cons} since $m-1$ is odd and $\alpha$ is as in \Cref{prop:ORS}.

    For (2), we induct on $n$.
    In the case that $n=0$, then $m=0$ so $\alpha_m=\alpha_0=0/1$ and $\alpha_{m-1}=\alpha_{-1}=1/0$. Then
    \[ \cT(\alpha_{-1}\oplus^{\pm 1}\alpha_0)(1,z)=\cT(\pm 1)(1,z)=\pm z. \]

    Now assume that $\cT(\alpha_{m-3}\oplus^{\pm1} \alpha_{m-2})(1,z)=(-1)^{\tau}z$ for some $\tau\in\{0,1\}$.
    By \Cref{lem:farey_truncs}, $\alpha_m=\alpha_{m-2}\oplus^{\pm 2\zeta_n}\alpha_{m-1}$ and $\alpha_{m-1}=\alpha_{m-3}\oplus^{\pm 3\sigma_n}\alpha_{m-2}$.
    By \Cref{lem:cons}, since $m-1$ is odd, $\cT(\alpha_{m-1})(1,z)=0$. Applying Farey recursion with the Farey pair $\{\alpha_{m-1},\alpha_{m-2}\}$ and using $\alpha_m=\alpha_{m-2}\oplus^{\pm 2\zeta_n}\alpha_{m-1}$, we get from \Cref{lem:farey_truncs} that
    \[ \cT(\alpha_{m-1}\oplus^{\pm1}\alpha_m)(1,z)= -\cT(\alpha_{m-2}\oplus^{\pm2|\zeta_n|-1}\alpha_{m-1})(1,z). \]
    Applying Farey recursion repeatedly and using $\cT(\alpha_{m-1})(1,z)=0$ to reduce by one term at each step, we get that this equals
    \[ \pm \cT (\alpha_{m-2}\oplus^{\pm1}\alpha_{m-1})(1,z). \]

    By \Cref{lem:cons}, since $m-2$ is even, $\cT(\alpha_{m-2})(1,z)=(-1)^{\sum_{i=1}^{n-1}\zeta_i}$. Applying Farey recursion with the Farey pair $\{\alpha_{m-2},\alpha_{m-3}\}$ and using $\alpha_{m-1}=\alpha_{m-3}\oplus^{\pm 3\zeta_n}\alpha_{m-2}$, we get from \Cref{lem:farey_truncs} that
    \[ \cT(\alpha_{m-2}\oplus^{\pm1}\alpha_{m-1})(1,z)= -\cT(\alpha_{m-3}\oplus^{\pm3|\sigma_n|-1}\alpha_{m-2})(1,z). \]
    Applying Farey recursion repeatedly, we have at each step
    \begin{multline*}
        \cT(\alpha_{m-3}\oplus^{\pm k}\alpha_{m-2})(1,z)
     = (-1)^{\sum_{i=1}^{n-1}\zeta_i} \cT(\alpha_{m-3}\oplus^{\pm (k-1)}\alpha_{m-2})(1,z) 
     \\ - \cT(\alpha_{m-3}\oplus^{\pm (k-2)}\alpha_{m-2})(1,z) .
    \end{multline*}
     Since $3|\sigma_n|$ is odd, after an even number of reductions, we get that this equals
     \[ \pm \cT (\alpha_{m-3}\oplus^{\pm1}\alpha_{m-2})(1,z) \]
     and by induction this equals $\pm z$.

     Finally, (3) also follows from induction on $\ell$. The case $\ell=1$ is exactly (1) and (2). Assume that (3) holds for $\ell-1$. 
     Applying the Farey recursion at $w-1$ three consecutive times, and using the inductive assumption together with \Cref{lem:cons}, gives
     \[ \cT(\alpha_{m-1}\oplus^{\pm (3\ell-3)}\alpha_m)(1,z)= 0 \]
     and 
     \[ \cT(\alpha_{m-1}\oplus^{\pm (3\ell-2)}\alpha_m)(1,z)= (-1)^{\tau+(\ell-1)\left( 1+\sum_{i=1}^n \zeta_i \right)}z . \qedhere \]
\end{proof}

\begin{proposition}\label{prop:recursion_epsilon}
Let $m=2n$. For $\alpha=\alpha_{m+1}$ as in \Cref{prop:ORS}, we have
\[ \frac{\cT(\alpha)}{z(w^2-1)}\Big|_{w=1} = \pm \frac{\cT(\alpha_{m-1})}{z(w^2-1)}\Big|_{w=1} \pm \sigma_{n+1} \frac{\cT(\alpha_m)^2-1}{w^2-1}(1,z) . \]
\end{proposition}

\begin{proof}
    It suffices to show that $(w^2-1)^2$ divides the expression
    \[ \cT(\alpha)\pm \cT(\alpha_{m-1})\pm z\sigma_{n+1}(\cT(\alpha_{m})^2-1)
    \]
    since by \Cref{lem:Z[z^2]}, dividing by $z(w^2-1)$ and setting $w=1$ gives the result.

    Note first that by \Cref{lem:farey_truncs}, $\alpha =\alpha_{m-1}\oplus^{\pm 3\sigma_{n+1}} \alpha_{m}$, so applying the Farey recursive function gives
    \begin{multline*}
        \cT(\alpha_{m+1}) 
        = (\cT(\alpha_m)^2-1) \cdot \cT(\alpha_{m-1}\oplus^{\pm (3\ell-2)}\alpha_m)
        \\ - \cT(\alpha_m)\cdot \cT(\alpha_{m-1}\oplus^{\pm (3\ell-3)}\alpha_m)
    \end{multline*}
    where $\ell:=|\sigma_{n+1}|$.
    
    Define
    \[ f_1(\ell)=\ell\left( 1+\sum_{i=1}^n \zeta_i \right) \quad \text{ and } \quad f_2(\ell)=\tau+(\ell-1)\left( 1+\sum_{i=1}^n \zeta_i \right) \]
    where $\tau\in\{0,1\}$ is as in \Cref{lem:Tmu}.

    We induct on $\ell$ to show that $(w^2-1)^2$ divides 
    \[ \cT(\alpha_{m+1}) - (-1)^{f_1(\ell)} \cT(\alpha_{m-1}) -(-1)^{f_2(\ell)} z \ell (\cT(\alpha_{m})^2-1)
    \]

    When $\ell=1$, the Farey recursion expansion gives
    \begin{multline*}
        \cT(\alpha_{m+1}) - (-1)^{f_1(1)} \cT(\alpha_{m-1}) - (-1)^{f_2(1)} z(\cT(\alpha_m)^2-1)
        \\ = \left( \cT(\alpha_{m-1}\oplus^{\pm 1}\alpha_m) - (-1)^\tau z \right) \cdot (\cT(\alpha_m)-1) 
        \\ - \cT(\alpha_{m-1}) \cdot \left( \cT(\alpha_m)+(-1)^{1+\sum\zeta_i} \right)
    \end{multline*} 
    By \Cref{lem:Tmu}, $\cT(\alpha_{m-1}\oplus^{\pm 1}\alpha_m) - (-1)^\tau z$ vanishes at $w=1$, so it is divisible by $(w^2-1)$.
    By \Cref{lem:cons}, both $\cT(\alpha_{m-1})(1,z)$ and $\cT(\alpha_m)+(-1)^{1+\sum\zeta_i}$ also vanish, and hence are divisible by $(w^2-1)$. So the entire expression is divisible by $(w^2-1)^2$. 

    Now assume the result holds $\ell-1$.
    Note that $\cT(\alpha_{m-1}\oplus^{\pm (3\ell-3)}\alpha_m)$ corresponds to the case $|\sigma_{n+1}|=\ell-1$. So modulo $(w^2-1)^2$,
    \begin{multline*}
        \cT(\alpha_{m-1}\oplus^{\pm (3\ell-3)}\alpha_m) 
        \\ \equiv (-1)^{f_1(\ell-1)} \cT(\alpha_{m-1}) + (-1)^{f_2(\ell-1)} z (\ell-1) (\cT(\alpha_{m})^2-1) .
    \end{multline*}  
    
    We have
    \begin{multline*}
        \cT(\alpha_{m+1})-(-1)^{f_1(\ell)}\cT(\alpha_{m-1})-(-1)^{f_2(\ell)}z\ell (\cT(\alpha_m)^2-1)
        \\ = (\cT(\alpha_m)^2-1)\cdot \left( \cT(\alpha_{m-1}\oplus^{\pm (3\ell-2)}\alpha_m) - (-1)^{f_2(\ell)}z\ell \right)
        \\ - \cT(\alpha_m)\cdot \cT(\alpha_{m-1}\oplus^{\pm (3\ell-3)}\alpha_m) - (-1)^{f_1(\ell)}\cT(\alpha_{m-1}) .
    \end{multline*} 
    so replacing $\cT(\alpha_{m-1}\oplus^{\pm (3\ell-3)}\alpha_m) $ by its congruence and regrouping we get
    \begin{multline*}
        \equiv (\cT(\alpha_m)^2-1)\cdot \left[ \cT(\alpha_{m-1}\oplus^{\pm (3\ell-2)}\alpha_m) - (-1)^{f_2(\ell)}z\ell \right.
        \\ \left. + (-1)^{f_2(\ell-1)} z (\ell-1) \cT(\alpha_{m})\right]
        \\ - \cT(\alpha_{m-1})\left( (-1)^{f_1(\ell)} + (-1)^{f_1(\ell-1)} \cT(\alpha_m) \right) .
    \end{multline*} 
    By \Cref{lem:cons}, $(\cT(\alpha_m)^2-1)$ is divisible by $(w^2-1)$. By \Cref{lem:Tmu} and also \Cref{lem:cons}, the term inside the square brackets is also divisible by $(w^2-1)$. So the first product is divisible by $(w^2-1)^2$. 
    By \Cref{lem:cons}, $\cT(\alpha_{m-1})$ is divisible by $(w^2-1)$ and $(-1)^{f_1(\ell)} + (-1)^{f_1(\ell-1)} \cT(\alpha_m) $ vanishes at $w=1$. So the second product is also divisible by $(w^2-1)^2$, and the result follows.
\end{proof}

\begin{proposition}\label{prop:recursion_z}
With $\alpha$ as in \Cref{prop:ORS} and $m=2n$, let  
\[ \beta = [3\sigma_1,2\zeta_1,\dots,3\sigma_{n+1},2\zeta_{n+1}] \]
be the continued fraction expansion obtained from $\alpha$ by appending the term $2\zeta_{n+1}$.
Then 
    \[
    \frac{\cT(\beta)^2-1}{w^2-1}(1,z) = \frac{\cT(\alpha_m)^2-1}{w^2-1}(1,z)\pm 2\zeta_{n+1}z^2 \cdot \frac{\cT(\alpha)}{z(w^2-1)}\Big|_{w=1}
    \]
where $\frac{\cT(\alpha)}{z(w^2-1)}\Big|_{w=1}\in\ZZ[z^2]$ by \Cref{lem:Z[z^2]}.
\end{proposition}

\begin{proof}
    By \Cref{lem:farey_truncs} $\beta=\alpha\oplus^{\pm 2\zeta_{n+1}} \alpha_m$, so
    \[ \cT(\beta) = \cT(\alpha) \cdot \cT(\alpha_m\oplus^{\pm (2s-1)} \alpha)-\cT(\alpha_m\oplus^{\pm (2s-2)} \alpha)
    \]
    where $s=|\zeta_{n+1}|$.

    Denote $I_k=\frac{\cT(\alpha_m\oplus^{\pm k}\alpha)^2-1}{w^2-1}(1,z)$. It suffices to show that
    \[ I_{2s} = I_0 \pm 2\zeta_{n+1}z^2 \cdot \frac{\cT(\alpha)}{z(w^2-1)}\Big|_{w=1}\]
    since $I_{2s}=\frac{\cT(\beta)^2-1}{w^2-1}(1,z)$ and $I_0=\frac{\cT(\alpha_m)^2-1}{w^2-1}(1,z)$.

    By squaring the Farey recursion expansion and dividing by $w^2-1$, we get
    \begin{multline*}
        I_{2s} 
         = \frac{\cT(\alpha)^2\cdot \cT(\alpha_m\oplus^{\pm (2s-1)} \alpha)^2}{w^2-1}(1,z)
        \\ - 2\cdot \frac{\cT(\alpha) \cdot \cT(\alpha_m\oplus^{\pm (2s-1)} \alpha) \cdot \cT(\alpha_m\oplus^{\pm (2s-2)} \alpha)}{w^2-1}(1,z)
        + I_{2s-2}.
    \end{multline*}
    
    Since $\cT(\alpha)(1,z)=0$ by \Cref{lem:cons}, the first term vanishes.
    
    By \Cref{lem:Z[z^2]}, $\frac{\cT(\alpha)}{(w^2-1)}(1,z)=z\cdot \frac{\cT(\alpha)}{z(w^2-1)}\Big|_{w=1}$, and by \Cref{lem:Tmu},
    \[ \cT(\alpha_m\oplus^{\pm (2s-1)} \alpha)(1,z)=(-1)^{\tau+\zeta_{n+1}-1}z \]
    and
    \[ \cT(\alpha_m\oplus^{\pm (2s-2)} \alpha)(1,z)=(-1)^{\sum_{i=1}^n\zeta_i} . \]
    Therefore the second term becomes
    \begin{multline*}
        -2\cdot (-1)^{\tau+\zeta_{n+1}-1} z 
        \cdot (-1)^{\sum_{i=1}^n\zeta_i} z
        \cdot \frac{\cT(\alpha)}{z(w^2-1)}\Big|_{w=1} 
        \\ = \pm 2 z^2 \cdot \frac{\cT(\alpha)}{z(w^2-1)}\Big|_{w=1} .
    \end{multline*}

    We now have
    \[ I_{2s} = I_{2s-2}
        \pm 2 z^2 \cdot \frac{\cT(\alpha)}{z(w^2-1)}\Big|_{w=1} . \]

    When $s=1$, we already have $I_{2s} = I_0
    \pm 2 z^2 \cdot \frac{\cT(\alpha)}{z(w^2-1)}\Big|_{w=1}$ from above.

    Assume $I_{2(s-1)} = I_0 \pm 2(\zeta_{n+1} \mp 1) \cdot \frac{\cT(\alpha)}{z(w^2-1)}\Big|_{w=1}$.
    Then
    \[ I_{2s} = I_{2s-2}
        \pm 2 z^2 \cdot \frac{\cT(\alpha)}{z(w^2-1)}\Big|_{w=1} 
        = I_0 \pm 2\zeta_{n+1} \cdot \frac{\cT(\alpha)}{z(w^2-1)}\Big|_{w=1} . \qedhere \]
\end{proof}

\subsection{Proof of Theorem A}
We are now ready to prove \Cref{thm:A}.

\begin{proof}[Proof of \Cref{thm:A}]
For $\ell=2k$ with $0\leq k\leq n$, define
\[ a_\ell:=\frac{\cT(\alpha_{\ell-1})}{z(w^2-1)}\Big|_{w=1}  \quad \text{ and } \quad b_\ell:=\frac{\cT(\alpha_\ell)^2-1}{w^2-1}(1,z).\]
Recall that by \Cref{lem:Z[z^2]}, $a_\ell\in\ZZ[z^2]$.
We also have that $b_\ell\in\ZZ[z^2]$, since $\cT(\alpha_\ell)\in\ZZ[w^2,z^2]$ by \Cref{lem:T0} and $\cT(\alpha_\ell)(1,z)=\pm1$ by \Cref{lem:cons}. 

By \Cref{prop:recursion_epsilon,prop:recursion_z}, we have the recursive relations
\[ a_{\ell+2}=\pm a_\ell \pm \sigma_{k+1} b_\ell \quad \text{ and } \quad b_{\ell+2} = b_\ell + 2\zeta_{k+1}z^2 a_{\ell+2} . \]

When $\ell=0$, $a_0=\frac{\cT(\alpha_{-1}}{z(w^2-1)}\Big|_{w=1}$, where $\alpha_{-1}=1/0$. 
By definition, $\cT(1/0)=0$, so $a_0=0$.
Also, $b_0=\frac{\cT(\alpha_{0})^2-1}{w^2-1}(1,z)$, where $\alpha_0=0/1$.
By definition, $\cT(0/1)=w$, so $b_0=1$.

We now claim that when $\ell\geq2$,
\begin{gather*} \deg a_\ell = \ell-2 ,\quad \deg b_\ell = \ell 
,\quad a_\ell\equiv \sum_{i=1}^k\sigma_i\bmod {2\ZZ[z]} 
, \\ \text{and } \quad b_\ell\equiv 1\bmod {2\ZZ[z]} . \end{gather*}

When $\ell=2$, $a_2=\frac{\cT(\alpha_{1})}{z(w^2-1)}\Big|_{w=1}$, where $\alpha_{1}=[3\sigma_1]$.
By \Cref{lem:farey_truncs}, $[3\sigma_1]=1/0\oplus^{\pm3|\sigma_1|}0/1$. Applying Farey recursion repeatedly with $\cT(1/0)=0$ and $\cT(0/1)=w$, we get by induction on $|\sigma_1|$ that $\cT([\sigma_1])=\pm\sigma_1 z(w^2-1)$, and hence $a_2=\pm\sigma_1$. So $\deg a_2=0$ and $a_2 \equiv \sigma_1 \bmod 2\ZZ[z^2]$.
By \Cref{prop:recursion_z} with $\beta=[3\sigma_1,2\zeta_1]$, $\alpha=[3\sigma_1]$, and $\alpha_0=0/1$,
\[ b_2 
= \frac{\cT(0/1)^2-1}{w^2-1}(1,z) \pm 2\zeta_1 z^2\cdot \frac{\cT([3\sigma_1])}{z(w^2-1)}\Big|_{w=1}
= \frac{w^2-1}{w^2-1}(1,z) \pm 2 \zeta_1 \sigma_1 z^2 . \]
So $\deg b_2=2$ and $b_2\equiv 1\bmod 2\ZZ[z^2]$.

Now assume the claim holds for $\ell-2$, that is, $\deg a_{\ell-2} = \ell-4$, $\deg b_{\ell-2} = \ell-2$, $a_{\ell-2}\equiv \sum_{i=1}^{k-1}\sigma_i\bmod {2\ZZ[z]}$, and $b_{\ell-2} \equiv 1\bmod {2\ZZ[z]}$.
By \Cref{prop:recursion_epsilon}, $a_\ell=\pm a_{\ell-2}\pm \sigma_k b_{\ell-2}$. It follows that $\deg a_\ell =\deg b_{\ell-2}=\ell-2$ and 
\[ a_\ell\equiv a_{\ell-2} + \sigma_k b_{\ell-2} \equiv \sum_{i=1}^{k-1}\sigma_i +\sigma_k \equiv \sum_{i=1}^{k}\sigma_i \bmod 2\ZZ[z^2]. \]
By \Cref{prop:recursion_z},
$b_\ell= b_{\ell-2}\pm 2\zeta_k z^2 a_\ell$. It follows that $\deg b_\ell = 2+\deg a_\ell=\ell$ and 
\[ b_\ell \equiv b_{\ell-2} + 2\zeta_k z^2 a_\ell \equiv b_{\ell-2}\equiv 1 \bmod 2 \ZZ[z^2] . \]
This confirms the claim.

Now, by \Cref{lem:Z[z^2]}, 
$\phi_\alpha(x) = a_{2n+2} $ after substituting $z^2=x^2-3$, which, by the claim, satisfies $\deg \phi_\alpha = 2n$ and $\phi_\alpha\equiv \sum_{i=1}^{n+1}\sigma_i\bmod 2\ZZ[x]$.
Then by \Cref{lem:epsumodd}, $\phi_\alpha\equiv 1\bmod 2\ZZ[x]$. In particular, except for the constant term which is odd, all other coefficients of $\phi_\alpha$ are even. 

We claim that $\phi_\alpha$ cannot be written as a product $\alpha=fg$ where $f,g\in\ZZ[x]$ and either $f$ or $g$ has degree $\geq 1$ and has odd leading term coefficient.
Suppose by contradiction that this is the case; then we can write $\phi_\alpha=\sum_{i=0}^{2n} c_ix^i$, $f=\sum_{i=0}^m p_ix^i$ where $p_m$ is odd, and $g=\sum_{i=0}^{2n-m} q_ix^i$.
Let $j=\max\{0\leq i\leq {2n-m} \mid q_i \text{ is odd}\}$, which exists since otherwise the constant term of $\phi_\alpha$ would be even. 
Consider the coefficient of $\phi_\alpha=fg$ in degree $j+m\ge 1$:
\[ c_{j+m} = \sum_{(i,k)=j+m}p_i q_k=\sum_{k\geq j}p_{j+m-k} q_k. \]
Since $c_{j+m}$ is even, and $q_k$ is even for $k>j$, we conclude that $p_m q_j$ is also even, contradicting the fact that both $p_m$ and $q_j$ are odd.

For any intersection point $(x_0,y_0)$ between $X_0(\Gamma_{\alpha})$ and $x^2-y-1=0$, we always have $\phi_\alpha(x_0)=0$.
If $x_0$ were an algebraic integer, then its minimal polynomial $q(x)\in\ZZ[x]$ would divide $\phi_\alpha$. It cannot be that $\phi_\alpha=q$ since one has even leading coefficient and the other has leading coefficient 1. This would then contradict the claim. Hence $x_0$ is not an algebraic integer, and the corresponding character is ANI.
\end{proof}

%% file: Slopes.tex
\section{Detected boundary slopes}
\label{sec:slope}

Having established that the intersection points in \Cref{thm:A} are ANI characters, we now determine the boundary slopes of the essential surfaces they detect.
From work of Schanuel and Zhang (see \Cref{boundary_slope}), each such ANI character, with respect to a discrete valuation $\nu$,
detects an essential surface whose boundary slope $\delta$ satisfies $\nu(\chi(\delta))>0$. To identify this slope, we explicitly compute the representations of the meridian $\mu$ and the longitude $\lambda$, to show in \Cref{thm:B} which slope satisfies this condition.

\subsection{Combinatorics of continued fractions}

We first introduce some notations used extensively in this section.

\begin{definition}\label{notation_boundary_slope}
For the rest of this section, we fix the following notations. Let $\alpha\in \hat\QQ$ be a continued fraction of the form $\alpha=q/p=[a_1,a_2,\dots,a_m]$ where $a_i\in \ZZ$, $|a_i|\geq 2$, and $a_1\ge 2$. For every $1\leq j\leq m$, let $\alpha_j=q_j/p_j$ denote the continued fraction consisting of the first $j$ terms of $\alpha$. We also define the sequence $L_j$ with length $q_j$: for $1\le i\le q_j$,
\[
L_j(i)=\#\{0\le k<p_j\mid \left\lfloor \frac{kq_j}{p_j}\right\rfloor=i-1\}.
\]

We define $\cdot*\cdot$ to be the concatenation of sequences, i.e. \[
(b_1,b_2,\cdots,b_k)*(c_1,c_2,\cdots,c_l)=(b_1,b_2,\cdots,b_k,c_1,c_2,\cdots,c_l).\]
We may write $n\times S$ when concatenating $n$ copies of the same sequence $S$.

If the sequence $S_1=(b_1,\cdots,b_k)$ appears as an initial subsequence of $S_2$, i.e. $S_2= (b_1,\dots,b_k, b_{k+1},\dots b_l)$, we define
\[ S_2-S_1 =  (b_{k+1},\dots b_l) . \]

If the sequence $S_1=(b_{k+1},\cdots,b_l)$ appears as a terminal subsequence of $S_2$, i.e. $S_2= (b_1,\dots,b_k, b_{k+1},\dots b_l)$, we define
\[ -S_1+S_2 =  (b_1,\dots b_k) . \]
\end{definition}

\begin{lemma}\label{lem:qp>0andq_j+1>q_j}
    Let $\alpha$ be as in \Cref{notation_boundary_slope}. Then for all $j$, 
    \[ 0\le q_j<p_j  \quad \text{ and } \quad q_{j+1}>q_j,p_{j+1}>p_j. \]
\end{lemma}

\begin{proof}
    We induct on $j$. When $j=0$, we have $q_0/p_0=0/1$ and $q_{1}/p_{1}=1/a_1$ where $a_1\geq2$, so $0=q_0<p_0=1$ and $q_1=1<a_1=p_1$.

    Assume the claim holds for all indices less than $j$.
    By \Cref{lem:farey_truncs} we have
    \[
    \frac {q_j}{p_j}=\frac{q_{j-2}}{p_{j-2}}\oplus^{\pm a_j}\frac{q_{j-1}}{p_{j-1}}
    \]
    so $q_j=q_{j-2}\pm a_j q_{j-1}$ and $p_j=p_{j-2}\pm a_j p_{j-1}$. 
    By assumption, $q_{j-1}>q_{j-2}\geq 0$ and $p_{j-1}>p_{j-2}\geq 0$.
    Since $|a_j|\geq 2$,
    \[ |a_j| q_{j-1}\geq 2 q_{j-1} > 2 q_{j-2} > q_{j-2} \]
    so $q_{j-1}\pm a_j q_{j-1}$ is nonzero and has the same sign as $\pm a_j$, giving $q_j>0$. The same argument applies to $p_j$, so also $p_j>0$.

    We also have 
    \[ q_j= |q_{j-2}\pm a_j q_{j-1} | \geq |a_j|q_{j-1} - q_{j-2} \geq 2 q_{j-1} - q_{j-2} > q_{j-1} \]
    and similarly $p_j>p_{j-1}$.
\end{proof}

\begin{lemma}\label{lem:farey_sum_oplus_pm}
    Let $\epsilon=1$ when $a_{m-1},a_{m}$ have the same sign and $\epsilon=-1$ otherwise, then
    \[
    q=|a_m|q_{m-1}+\epsilon q_{m-2},\qquad p=|a_m|p_{m-1}+\epsilon p_{m-2}.
    \]
\end{lemma}

\begin{proof}
By \Cref{lem:farey_truncs}, we know for some $c_1, c_2\in \{\pm 1\}$,
\[
\begin{bmatrix}
    q\\ p
\end{bmatrix}=\pm\begin{bmatrix}
     c_1 q_{m-3} & c_2 q_{m-2}\\ c_1 p_{m-3} & c_2 q_{m-2}
\end{bmatrix}\begin{bmatrix}
    0&1\\ 1&a_{m-1}
\end{bmatrix}\begin{bmatrix}
    0&1\\ 1&a_m
\end{bmatrix}\begin{bmatrix}
    0\\ 1
\end{bmatrix}.
\]
We then compute that
\begin{multline*}
    \begin{bmatrix}
    c_1 q_{m-3} & c_2 q_{m-2}\\ c_1 p_{m-3} & c_2 q_{m-2}
\end{bmatrix} \begin{bmatrix}
    0&1\\ 1&a_{m-1}
\end{bmatrix} \begin{bmatrix}
    0&1\\ 1&a_m
\end{bmatrix} \\
= \begin{bmatrix}
    c_1 q_{m-3}+a_{m-1} c_2 q_{m-2} & a_m(c_1 q_{m-3}+a_{m-1}c_2q_{m-2}) +c_2 q_{m-2}\\ c_1 p_{m-3}+a_{m-1}c_2 p_{m-2} & a_m(c_1 p_{m-3} +a_{m-1} c_2 p_{m-2})+c_2 p_{m-2}
\end{bmatrix}.
\end{multline*}
By \Cref{lem:qp>0andq_j+1>q_j}, $q_{m-1}=d_{1}c_2 [c_1 q_{m-3}+a_{m-1}c_2 q_{m-2}]$ where $d_1=1$ when $a_{m-1}$ is positive and $d_1=-1$ otherwise. Define $d_2=1$ when $a_m$ is positive and $d_2=-1$ otherwise.
Then
\[ a_m(c_1 q_{m-3}+a_{m-1}c_2 q_{m-2}) +c_2 q_{m-2} =c_2 [d_1 d_2 |a_m| q_{m-1}+q_{m-2}], \]
and the same holds for $p$, so the conclusion follows since $\epsilon=d_1 d_2$.
\end{proof}

\begin{lemma}\label{qp-pq=pm1}
    When $m$ is odd and $a_{m}<0$ or $m$ is even and $a_{m}>0$, $q_{m-1}p-qp_{m-1}=1$. Otherwise, $q_{m-1}p-qp_{m-1}=-1$.
\end{lemma}

\begin{proof}
    We induct on $m$. When $m=1$, we have $q/p=1/a_m$ given $a_1>0$ and $q_{m-1}/p_{m-1}=0/1$, and the claim holds. According to \Cref{lem:farey_sum_oplus_pm},
    when $m=2$, if $a_2>0$, we have $q/p=\frac{a_2}{a_1a_2+1}$ and $q_1p-qp_1=1\cdot(a_1a_2+1)-a_2\cdot a_1=1$; if $a_2<0$, we have $q/p=\frac{a_2}{a_1a_2-1}$ and $q_1p-qp_1=1\cdot(a_1a_2-1)-a_2\cdot a_1=-1$.

    Now assume $m$ is odd and $a_m<0$. If $a_{m-1}<0$, by induction hypothesis we know $q_{m-2}p_{m-1}-q_{m-1}p_{m-2}=-1$, and by \Cref{lem:farey_sum_oplus_pm} we have $q=|a_m|q_{m-1}+q_{m-2}$ and same for $p$. Then we may compute that
    \begin{align*}
        q_{m-1}p-qp_{m-1}&=q_{m-1}(|a_m|p_{m-1}+p_{m-2})-(|a_m|q_{m-1}+q_{m-2})p_{m-1}\\
        &=q_{m-1}p_{m-2}-q_{m-2}p_{m-1}=1.
    \end{align*}
    Similarly, if $a_{m-1}>0$, then $q_{m-2}p_{m-1}-q_{m-1}p_{m-2}=1$ and $q=|a_m|q_{m-1}-q_{m-2}$, so we have
    \begin{align*}
        q_{m-1}p-qp_{m-1}&=q_{m-1}(|a_m|p_{m-1}-p_{m-2})-(|a_m|q_{m-1}-q_{m-2})p_{m-1}\\
        &=-q_{m-1}p_{m-2}+q_{m-2}p_{m-1}=1.
    \end{align*}

    The cases when $m$ is odd and $a_m>0$ or when $m$ is even can be proven similarly.
\end{proof}

\begin{proposition}\label{prop:recursionLj}
Let $\alpha=q/p=[a_1,a_2,\dots,a_{m},a_{m+1}]$ be as in \Cref{notation_boundary_slope}.
We have the following recursive relations in $L_j$.
\begin{enumerate}
    \item When $m$ is even:
    \begin{enumerate}
        \item $a_{m},a_{m+1}>0$: $L_{m+1}=(a_{m+1}\times L_{m})*L_{m-1}$
        \item $a_{m},a_{m+1}<0$:
        $L_{m+1}=L_{m-1}*(|a_{m+1}|\times L_{m})$
        \item $a_{m}>0,a_{m+1}<0$:
        $L_{m}=(L_{m}-L_{m-1})*[(|a_{m+1}|-1)\times L_{m}]$
        \item $a_{m}<0,a_{m+1}>0$:
        $L_{m+1}=[(a_{m+1}-1)\times L_{m}]*(-L_{m-1}+L_{m})$
    \end{enumerate}
    \item When $m$ is odd:
    \begin{enumerate}
        \item $a_{m},a_{m+1}>0$: $L_{m+1}=L_{m-1}*(a_{m+1}\times L_{m})$
        \item $a_{m},a_{m+1}<0$:
        $L_{m+1}=(|a_{m+1}|\times L_{m})*L_{m-1}$
        \item $a_{m}>0,a_{m+1}<0$:
        $L_{m+1}=[(|a_{m+1}|-1)\times L_{m}]*(-L_{m-1}+L_{m})$
        \item $a_{m}<0,a_{m+1}>0$:
        $L_{m+1}=(L_{m}-L_{m-1})*[(a_{m+1}-1)\times L_{m}]$.
    \end{enumerate}
\end{enumerate}    
\end{proposition}

\begin{proof}
    Let $A:=|a_{m+1}|$.
    According to \Cref{lem:farey_sum_oplus_pm} and \Cref{qp-pq=pm1}, we may prove this by cases.

    We will prove the cases (1)(a) and (2)(b). Other cases have similar arguments. 
    
    In cases (1)(a) and (2)(b), we have $q=Aq_{m}+q_{m-1}$, $p=Ap_{m}+p_{m-1}$ and $q_{m-1}p_m-q_mp_{m-1}=1$.

        We prove that for $0\le a<A$, $0\le b<p_m$, and $0\le c<p_{m-1}$,
        \begin{equation} \label{1a_and_2b} \begin{gathered}
        \left\lfloor \frac{(ap_m+b)q}{p}\right\rfloor=aq_m+\left\lfloor \frac{bq_m}{p_m}\right\rfloor
        \\ \text{and }\quad \left\lfloor \frac{(Ap_m+c)q}{p}\right\rfloor=Aq_m+\left\lfloor \frac{cq_{m-1}}{p_{m-1}}\right\rfloor
        \end{gathered} \end{equation}
        Suppose $n\le bq_m/p_m<n+1$. Then we have
        \begin{align*}
            (Aq_m+q_{m-1})(ap_m+b)&-(aq_m+n)(Ap_m+p_{m-1})
            \\ &=a(q_{m-1}p_m-q_mp_{m-1})+bq-np
            \\ &=a+bq-np
            \\ &\ge a+bq-\frac{bq_m}{p_m}p
            \\ &=\frac{ap_m+bp_m(Aq_m+q_{m-1})-bq_m(Ap_m+p_{m-1})}{p_m}
            \\ &=\frac{ap_m+b(q_{m-1}p_m-q_mp_{m-1})}{p_m}
            \\ &=\frac{ap_m+b}{p_m}
            \\ &\ge 0.
        \end{align*}
        Because $a+bq-\frac{bq_m}{p_m}p=\frac{ap_m+b}{p_m}$, we also have
        \begin{align*}
            q(ap_m+b)-(aq_m+n)p &=[a+bq-np-(a+bq-\frac{bq_m}{p_m}p)]+\frac{ap_m+b}{p_m}
            \\ & =\frac{bq_m}{p_m}p-np+\frac{ap_m+b}{p_m}
            \\ &=\frac{(bq_m-np_m)p+(ap_m+b)}{p_m}
            \\ &<\frac{(p_m-1)p+(A-1)p_m+p_m}{p_m}
            \\ &=\frac{p_mp+Ap_m-(Ap_m+p_{m-1})}{p_m}
            \\ &<p.
        \end{align*}
        Therefore, we have
        \begin{gather*}
            0\le q(ap_m+b) -(aq_m+n)p<p \\ \text{and }\quad  aq_m+n\le \frac{(ap_m+b)q}{p}<aq_m+n+1,
        \end{gather*}
        which proves the first equation in \Cref{1a_and_2b}.

        Now suppose $n\le cq_{m-1}/p_{m-1}<n+1$. Then we have
        \begin{align*}\label{1a_and_2b}
            (Aq_m+q_{m-1})(A &p_m +c)-(Aq_m+n)(Ap_m+p_{m-1})
            \\ &=A(q_{m-1}p_m-q_mp_{m-1})+cq-np
            \\ &\ge A+cq-\frac{cq_{m-1}}{p_{m-1}}p
            \\ &=\frac{Ap_{m-1}+cp_{m-1}(Aq_m+q_{m-1})-cq_{m-1}(Ap_m+p_{m-1})}{p_{m-1}}
            \\ &=\frac{Ap_{m-1}+Ac(q_{m}p_{m-1}-q_{m-1}p_{m})}{p_{m-1}}
            \\ &=\frac{Ap_{m-1}-Ac}{p_{m-1}}
            \\ &\ge 0.
        \end{align*}
        Then because $A+cq-\frac{cq_{m-1}}{p_{m-1}}p=\frac{Ap_{m-1}-Ac}{p_{m-1}}$, we have
        \begin{align*}
            q(Ap_m+c)-(Aq_m+n)p &=\frac{cq_{m-1}}{p_{m-1}}p-np+\frac{Ap_{m-1}-Ac}{p_{m-1}}
            \\ &=\frac{(cq_{m-1}-np_{m-1})p+A(p_{m-1}-c)}{p_{m-1}}
            \\ &<\frac{(p_{m-1}-1)p+Ap_{m}}{p_{m-1}}
            \\ &=\frac{p_{m-1}p+Ap_m-(Ap_m+p_{m-1})}{p_{m-1}}
            \\ &<p.
        \end{align*}
        This completes the proof of the second equation in \Cref{1a_and_2b} in the cases (1)(a) and (2)(b). Other cases are similar.
        
    Then the fact that $L_{m-1}$ is a subsequence of $L_m$ at the beginning (or at the ending) in cases (1)(d) and (2)(c) (or in cases (1)(c) and (2)(d)) follows from the proof above.
\end{proof}

\subsection{Representations and matrix computations}

We introduce the following notations used to compute $\epsilon$ and the representation of the word $w$ defined in \Cref{knotgrouppres}.

\begin{definition}\label{notation_boundary_slope_2}
    Under the setting of \Cref{thm:A} and \Cref{notation_boundary_slope}, let $(x_\alpha,y_\alpha)$ be an intersection point of the irreducible component of the trefoil knot and the canonical component of the knot $\alpha=q/p=[a_1,a_2,\dots,a_m]$. Since the irreducible component of the trefoil knot corresponds to $y-x^2+1=0$, we must have $y_\alpha-x_\alpha^2+1=0$, $y_\alpha=x_\alpha^2-1$.

    A non-abelian representation with character $(x_\alpha,y_\alpha)$ is a conjugate in $\mathrm{SL}_2(\CC)$ to a representation $\rho$ with $A=\rho(a)$ and $B=\rho(b)$ given by
\[
A=\begin{bmatrix}
    \theta& 1\\ 0& \frac 1{\theta}
\end{bmatrix} \quad \text{ and } \quad B=\begin{bmatrix}
    \theta& 0\\ t& \frac 1{\theta}
\end{bmatrix}
\]
where $\theta+1/\theta=x_\alpha$ and $y_\alpha=2-t$. Also, $\theta$ and $t$ satisfy the relation
\[
t=2-y_\alpha = 2-(x_\alpha^2-1) = 1-\theta^2-1/\theta^2.
\]

Define
    \begin{gather*}
        w_a^+(\alpha_j)=a^{\zeta_j^+(0)}b^{\zeta_j^+(1)}\cdots (\text{$a$ or $b$})^{\zeta_j^+(p-1)} , \quad \zeta_j^+(i)=(-1)^{\left\lfloor\frac{iq_j}{p_j}\right\rfloor},
        \\ w_a^-(\alpha_j) , \text{ and }\quad \zeta_j^-(i)=-(-1)^{\left\lfloor\frac{iq_j}{p_j}\right\rfloor} .
    \end{gather*}
We define $w_b^\pm(\alpha_j)$ similarly, starting with $b^{\zeta_j^\pm(0)}$ instead of $a^{\zeta_j^\pm(0)}$. We also define $W_{a,b}^{\pm}(\alpha_j)=\rho(w_{a,b}^{\pm}(\alpha_j))$.
\end{definition}

\begin{proposition}\label{computation_of_epsilon}
    If $\alpha = [3\sigma_1,2\zeta_1,\dots,2\zeta_n,3\sigma_{n+1}]$ is as in \Cref{prop:ORS}, then
    \[
    \epsilon=4\#\{1\le i\le n\mid \zeta_i<0\}-4\#\{1\le i\le n+1\mid \sigma_i<0\}-6\sum_1^{n+1}\sigma_i+2.
    \]
\end{proposition}

\begin{proof}
    Let $m=2n$. 
    It suffices to prove that the exponent of $w_a^+(\alpha_{m+1})$ is $e(w_a^+(\alpha_{m+1}))=3\sum_1^{n+1}\sigma_i+2\#\{1\le i\le n+1\mid \sigma_i<0\}-2\#\{1\le i\le n\mid \zeta_i<0\}$. We also prove that $e(w_a^+(\alpha_{m}))=1$, then we may prove these two claims by induction using \Cref{prop:recursionLj}. Notice that $e(w_a^+(\alpha_j))=\sum_{\text{odd}\ i}L_j(i)-\sum_{\text{even}\ i}L_j(i)=-\sum_{i=1}^{q_j}(-1)^iL_j(i)$. 

    For the base case, we have $e(w_a^+(\alpha_{-1}))=0$ and $e(w_a^+(\alpha_0))=1$. Now suppose the two claims hold for $j\le m-1$, which is odd. Similar to proof of \Cref{lem:cons}, we know $q_{m-2}$ is even and $q_{m-1}$ is odd. According to \Cref{prop:recursionLj}, we split into four cases:
    \begin{enumerate}
        \item $\sigma_n,\zeta_n>0$:
        Because $L_{m-2}$ has even length $q_{m-2}$ and $L_{m-1}$ has odd length $q_{m-1}$, by \Cref{prop:recursionLj} we have
        \begin{align*}
        -\sum_{i=1}^{q_m}(-1)^iL_m(i) &=-\sum_{i=1}^{q_{m-2}}(-1)^iL_{m-2}(i)-\sum_{i=1}^{q_{m-1}}(-1)^{i+q_{m-2}}L_{m-1}(i)-\cdots
        \\ & \hspace{1.9in} -\sum_{i=1}^{q_{m-1}}(-1)^{i+2\zeta_nq_{m-2}}L_{m-1}(i)
        \\ &=1+e(w_a^+(\alpha_{m-1}))-e(w_a^+(\alpha_{m-1}))+\cdots -e(w_a^+(\alpha_{m-1}))
        \\ &=1.
        \end{align*}
        \item $\sigma_n,\zeta_n<0$: This case can be proven using a similar argument as above.
        \item $\sigma_n>0,\zeta_n<0$: By \Cref{prop:recursionLj} we have
        \begin{align*}
            -\sum_{i=1}^{q_m}(-1)^iL_m(i) &= e(w_a^+(\alpha_{m-1}))-e(w_a^+(\alpha_{m-1}))+\cdots +e(w_a^+(\alpha_{m-1}))
            \\ & \hspace{.9in} -\sum_{i=q_{m-2}+1}^{q_{m-1}}(-1)^{i+(2|\zeta_n|-1)q_{m-1}-q_{m-2}}L_{m-1}(i)
            \\ &=e(w_a^+(\alpha_{m-1}))+\sum_{i=1}^{q_{m-1}}(-1)^i L_{m-1}(i)-\sum_{i=1}^{q_{m-2}}(-1)^i L_{m-2}(i)
            \\ &=1.
        \end{align*}
        \item $\sigma_n<0,\zeta_n>0$: This case can be proven using a similar argument as above.
    \end{enumerate}

    Then suppose the claims hold for $j\le m$. We know $q_{m-1}$ is odd and $q_{m}$ is even.
    We also split into four cases:
    \begin{enumerate}
        \item $\zeta_n,\sigma_{n+1}>0$: By \Cref{prop:recursionLj}, we have
        \begin{align*}
            -\sum_{i=1}^{q}(-1)^iL_{m+1}(i) &=\sum_{1}^{3\sigma_{n+1}}e(w_a^+(\alpha_m))-\sum_{i=1}^{q_{m-1}}(-1)^{i+3\sigma_{n+1}q_m}L_{m-1}(i)
            \\ &=3\sigma_{n+1}+3\sum_{i=1}^{n}\sigma_i+2\#\{1\le i\le n\mid \sigma_i<0\}
            \\ & \hspace{1.25in} -2\#\{1\le i\le n-1\mid \zeta_i<0\}
            \\ &=3\sum_{i=1}^{n+1}\sigma_i+2\#\{1\le i\le n+1\mid \sigma_i<0\}
            \\ & \hspace{1.5in} -2\#\{1\le i\le n\mid \zeta_i<0\}.
        \end{align*}
        \item $\zeta_n,\sigma_{n+1}<0$: This case can be proven using a similar argument as above.
        \item $\zeta_n>0,\sigma_{n+1}<0$: By \Cref{prop:recursionLj}, we have
        \begin{align*}
            -\sum_{i=1}^{q}(-1)^iL_{m+1}(i)&=-\sum_{i=1}^{q_m-q_{m-1}}(-1)^iL_m-\sum_1^{3|\sigma_{n+1}|-1}e(w_a^+(\alpha_m))
            \\ &=e(w_a^+(\alpha_m))+e(w_a^+(\alpha_{m-1}))-(3|\sigma_{n+1}|-1)
            \\ & =3\sum_{i=1}^{n}\sigma_i+2\#\{1\le i\le n\mid \sigma_i<0\} \\ & \hspace{.9in} -2\#\{1\le i\le n-1\mid \zeta_i<0\}  +3\sigma_{n+1}+2
            \\  &=3\sum_{i=1}^{n+1}\sigma_i+2\#\{1\le i\le n+1\mid \sigma_i<0\} \\ & \hspace{2in} -2\#\{1\le i\le n\mid \zeta_i<0\}.
        \end{align*}
        \item $\zeta_n<0,\sigma_{n+1}>0$: This case can be proven using a similar argument as above. \qedhere
    \end{enumerate}
\end{proof}

\begin{proposition}\label{computation_of_w}
    Under the setting of \Cref{notation_boundary_slope_2}, for $\alpha=q/p=[3\sigma_1,2\zeta_1,\dots,2\zeta_n,3\sigma_{n+1}]$,
    let $k=\#\{1\le i\le n+1\mid \sigma_i<0\}-\#\{1\le i\le n\mid \zeta_i<0\}$. Then
    \begin{enumerate}
     \item When $\sum_1^{n+1}\sigma_i$ is even:
    \begin{gather*}
    W_{a}^+(\alpha)=\pm A^{2k},\quad W_a^-(\alpha)=\pm A^{-2k};
    \\ W_b^+(\alpha)=\pm B^{2k}, \quad W_b^-(\alpha)=\pm B^{-2k}. \end{gather*}
    \item When $\sum_1^{n+1}\sigma_i$ is odd:
    \begin{gather*}
    W_{a}^+(\alpha)=\pm ABAB^{2k} ,\quad W_a^-(\alpha)=\pm A^{-1}B^{-1}A^{-1}B^{-2k};
    \\ W_b^+(\alpha)=\pm BABA^{2k} ,\quad W_b^-(\alpha)=\pm B^{-1}A^{-1}B^{-1}A^{-2k}.
    \end{gather*}
\end{enumerate}
\end{proposition}

\begin{proof}
    Let $m=2n$.
    We prove the proposition together with the following auxiliary claim simultaneously:
    for $\alpha_m=[3\sigma_1,\dots,2\zeta_n]$, let $k=\#\{1\le i\le n\mid \sigma_i<0\}-\#\{1\le i\le n\mid \zeta_i<0\}$, then
    \begin{gather*}
    W_a^+(\alpha_m)=\pm A^{k+1}B^{-k} ,\quad W_a^-(\alpha_m)=\pm A^{-k-1}B^{k};
    \\ W_b^+(\alpha_m)=\pm B^{k+1}A^{-k} ,\quad W_b^-(\alpha_m)=\pm B^{-k-1}A^{k}.
    \end{gather*}
    
    We will prove both the claim and the proposition by induction using \Cref{prop:recursionLj}.
    Notice that 
    \[ W_{a}^+(\alpha_j)=\prod_{\text{odd}\ i}A^{\zeta_{j}^+(i)}\cdot \prod_{\text{even}\ i}B^{\zeta_j^+(i)}=\prod_{i=1}^{p_j}f(i)^{\zeta_j^+(i)} ,\]
    where $f(i)=A$ when $i$ is odd and $f(i)=B$ otherwise.

    For the base case, we have $W_{a,b}^\pm(\alpha_{-1})=\rho(w_{a,b}^\pm (1/0))=I$ and $W_{a}^+(\alpha_{0})=\rho(w_{a}^+ (0/1))=A$.
    By \Cref{lem:epsumodd} and similar reasoning as in the proof of \Cref{computation_of_epsilon}, we know $q_{m-2}$ is even, $q_{m-1}$ is odd; $p_{m-2}$ is odd, and $p_{m-1}\equiv \sum_{i=1}^n\sigma_i\pmod 2$. 

    The proof of $W_{a,b}^\pm (\alpha_{m})$ and $W_{a,b}^\pm (\alpha)$ are essentially the same, so we only prove $W_a^+$ in each case.
    
    First we prove the auxiliary claim, assuming both claims hold for all $j\leq m-1$. We extend to $\alpha_m$ by appending $2\zeta_n$ to $\alpha_{m-1}$, splitting into four cases according to the signs of $\sigma_n$ and $\zeta_n$:
    \begin{enumerate}
        \item $\sigma_n,\zeta_n>0$: By \Cref{prop:recursionLj} we have
        \begin{align*}
            \prod_{i=1}^{p_m} & f(i)^{\zeta_m^+(i)}
            \\ =& \prod_{i=1}^{p_{m-2}}f(i)^{\zeta_{m-2}^+(i)}\cdot \prod_{i=1}^{p_{m-1}}f(i+p_{m-2})^{(-1)^{q_{m-2}}\zeta_{m-1}^+(i)}\\
            &\cdot \prod_{i=1}^{p_{m-1}}f(i+p_{m-2}+p_{m-1})^{(-1)^{q_{m-2}+q_{m-1}}\zeta_{m-1}^+(i)}\\
            &\cdots \prod_{i=1}^{p_{m-1}}f(i+p_{m-2}+(2\zeta_n-1)p_{m-1})^{(-1)^{q_{m-2}+(2\zeta_n-1)q_{m-1}}\zeta_{m-1}^+(i)}.
        \end{align*}
        By induction hypothesis, if $p_{m-1}$ is even, for some $k$ we have
        \begin{align*}
            W_a^+(\alpha_m) &=W_a^+(\alpha_{m-2})\cdot (W_b^+(\alpha_{m-1})\cdot W_b^-(\alpha_{m-1}))^{\zeta_n}\\
             &=\pm W_a^+(\alpha_{m-2})\cdot (B^{2k}\cdot B^{-2k})^{\zeta_n}\\
             &=\pm A^{\#\{1\le i\le n\mid \sigma_i<0\}-\#\{1\le i\le n\mid \zeta_i<0\}+1}B^{\#\{1\le i\le n\mid \zeta_i<0\}-\#\{1\le i\le n\mid \sigma_i<0\}};
        \end{align*}
        if $p_{m-1}$ is odd, for some $k$ we have
        \begin{align*}
            W_a^+(\alpha_m) &=W_a^+(\alpha_{m-2})\cdot (W_b^+(\alpha_{m-1})\cdot W_a^-(\alpha_{m-1}))^{\zeta_n}\\
             &=\pm W_a^+(\alpha_{m-2})\cdot (BABA^{2k}\cdot A^{-1}B^{-1}A^{-1}B^{-2k})^{\zeta_n}=\pm W_a^+(\alpha_{m-2})
        \end{align*}
        because $BABA^{2k}\cdot A^{-1}B^{-1}A^{-1}B^{-2k}=\pm I$ for any $k$: a direct computation using $t=1-\theta^2-1\theta^2$ shows that $BAB=ABA$, so by induction on $l$, we have $ABA^{l}=B(ABA^{l-1})=B(B^{l-1}AB)=B^lAB$. Thus, in our case we have
        \begin{align*}
        BABA^{2k}\cdot A^{-1}B^{-1}A^{-1}B^{-2k} &=ABA^{2k+1}\cdot A^{-1}B^{-1}A^{-1}B^{-2k}
        \\ &=B^{2k+1}AB\cdot A^{-1}B^{-1}A^{-1}B^{-2k}
        \\ &=B^{2k}ABA\cdot A^{-1}B^{-1}A^{-1}B^{-2k}
        \\ &=I.
        \end{align*}

        \item $\sigma_{n},\zeta_n<0$: similar to (1).
        
        \item $\sigma_n>0,\zeta_n<0$: By \Cref{prop:recursionLj} we have
        \begin{align*}
            \prod_{i=1}^{p_m} & f(i)^{\zeta_m^+(i)}
            \\ = & \prod_{i=1}^{p_{m-1}}f(i)^{\zeta_{m-1}^+(i)}\cdot \prod_{i=1}^{p_{m-1}}f(i+p_{m-1})^{(-1)^{q_{m-1}}\zeta_{m-1}^+(i)}\\
            &\cdots \prod_{i=1}^{p_{m-1}}f(i+(2|\zeta_n|-2)p_{m-1})^{(-1)^{(2|\zeta_n|-2)q_{m-1}}\zeta_{m-1}^+(i)}\\
            &\cdot \prod_{i=p_{m-2}+1}^{p_{m-1}}f(i+(2|\zeta_n|-1)p_{m-1}-p_{m-2})^{(-1)^{(2|\zeta_n|-1)q_{m-1}+q_{m-2}}\zeta_{m-1}^+(i)}.
        \end{align*}
        By induction hypothesis, if $p_{m-1}$ is even, for $k=\#\{1\le i\le n-1\mid \sigma_i<0\}-\#\{1\le i\le n-1\mid \zeta_i<0\}$ we have
        \begin{align*}
            W_a^+(\alpha_m)&=(W_a^+(\alpha_{m-1})\cdot W_a^-(\alpha_{m-1}))^{|\zeta_n|-1}\cdot W_a^+(\alpha_{m-1})
            \\ & \hspace{2.4in}  \cdot (W_b^-(\alpha_{m-2})^{-1}\cdot W_b^-(\alpha_{m-1}))
            \\ &=\pm A^{2k}\cdot A^{-k}B^{k+1}\cdot B^{-2k}=\pm A^kB^{1-k}
            \\ &=\pm A^{\#\{1\le i\le n\mid \sigma_i<0\}-\#\{1\le i\le n\mid \zeta_i<0\}+1}B^{\#\{1\le i\le n\mid \zeta_i<0\}-\#\{1\le i\le n\mid \sigma_i<0\}};
        \end{align*}
        if $p_{m-1}$ is odd, for the same $k$ we have
        \begin{align*}
            W_a^+(\alpha_m)&=(W_a^+(\alpha_{m-1}\cdot W_b^-(\alpha_{m-1})))^{|\zeta_n|-1}\cdot W_a^+(\alpha_{m-1})
            \\ & \hspace{1.09in} \cdot(W_a^-(\alpha_{m-2})^{-1}\cdot W_a^-(\alpha_{m-1}))
            \\ &=\pm ABAB^{2k}\cdot B^{-k}A^{k+1}\cdot A^{-1}(B^{-1}A^{-1}B^{-2k})
            \\ &=\pm ABAB^kA^k(A^{-2k}B^{-1}A^{-1})
            \\ &=\pm (BAB^{k+1})(A^{-k}B^{-1}A^{-1})
            \\ &=\pm (A^{k+1}BA)(B^{-1}A^{-1}B^{-k})
            \\ &=\pm A^kB^{1-k}
        \end{align*}
        because $ABA=BAB$, $ABA^l=B^lAB$ as shown above, and also $ABABAB = -I$.
        \item $\sigma_{n}<0,\zeta_n>0$: similar to (3).
    \end{enumerate}

    We now prove the main claim for $\alpha=\alpha_{m+1}$. Assume both claims hold for $j\le m$. We know $q_{m-1}$ is odd, $q_m$ is even; and $p_{m-1}\equiv \sum_{i=1}^{n}\sigma_i\pmod 2$, $p_m$ is odd. We extend to $\alpha$ by appending $3\sigma_{n+1}$ to $\alpha_m$ and again split into four cases according to the signs of $\zeta_n$ and $\sigma_{n+1}$: 
    \begin{enumerate}
        \item $\zeta_n,\sigma_{n+1}>0$: By \Cref{prop:recursionLj}, we have
        \begin{align*}
            \prod_{i=1}^{p}f(i)^{\zeta_{m+1}^+(i)}=&\prod_{i=1}^{p_{m}}f(i)^{\zeta_m^+(i)}\cdot \prod_{i=1}^{p_{m}}f(i+p_{m})^{(-1)^{q_{m}}\zeta_m^+(i)}\\
            &\cdots \prod_{i=1}^{p_{m}}f(i+(3\sigma_{n+1}-1)p_{m})^{(-1)^{(3\sigma_{n+1}-1)q_{m}}\zeta_m^+(i)}\\
            &\cdot \prod_{i=1}^{p_{m-1}}f(i+3\sigma_{n+1}p_{m})^{(-1)^{3\sigma_{n+1}q_m}\zeta_{m-1}^+(i)}.
        \end{align*}
        When $\sigma_{n+1}$ is even, for $k=\#\{1\le i\le n\mid \sigma_i<0\}-\#\{1\le i\le n\mid \zeta_i<0\}$ we have
        \begin{align*}
        W_a^+(\alpha)=&(W_a^+(\alpha_m)\cdot W_b^+(\alpha_m))^{3\sigma_{n+1}/2}\cdot W_a^+(\alpha_{m-1})\\
        =&\pm A^k(AB)^{3\sigma_{n+1}/2}A^{-k}\cdot W_a^+(\alpha_{m-1})=W_a^+(\alpha_{m-1}),
        \end{align*}
        and notice that $\sum_{1}^{n+1}\sigma_{i}\equiv \sum_{1}^{n}\sigma_{i}\pmod 2$, and $\#\{1\le i\le n\mid \sigma_i<0\}-\#\{1\le i\le n-1\mid \zeta_i<0\}=\#\{1\le i\le n+1\mid \sigma_i<0\}-\#\{1\le i\le n\mid \zeta_i<0\}$.

        When $\sigma_{n+1}$ is odd, for the same $k$ we have
        \begin{align*}
            W_a^+(\alpha)&=(W_a^+(\alpha_m)\cdot W_b^+(\alpha_m))^{(3\sigma_{n+1}-1)/2}\cdot W_a^+(\alpha_m)\cdot W_{b}^{+}(\alpha_{m-1})\\
            &=\pm A^k(AB)^{(3\sigma_{n+1}+3)/2}(AB)^{-1}\cdot B^{1-k}\cdot W_{b}^{+}(\alpha_{m-1})\\
            &=\pm (A^kB^{-1}A^{-1})B^{1-k}\cdot W_{b}^{+}(\alpha_{m-1})
            \\ &= \pm(B^{-1}A^{-1}B^k)B^{1-k}\cdot W_{b}^{+}(\alpha_{m-1})\\
            &=\pm ABA\cdot W_{b}^{+}(\alpha_{m-1}).
        \end{align*}
        Therefore, when $p_{m-1}$ is even, $p$ is odd and
        \[
        W_a^+(\alpha)=\pm ABA\cdot W_b^+(\alpha_{m-1})=\pm ABA\cdot B^{2k};
        \]
        when $p_{m-1}$ is odd, $p$ is even and
        \[
        W_a^+(\alpha)=\pm ABA\cdot W_b^+(\alpha_{m-1})=\pm ABA\cdot BABA^{2k}=\pm A^{2k}.
        \]
        
        \item $\zeta_n,\sigma_{n+1}<0$: similar to (1).
        
        \item $\zeta_n>0,\sigma_{n+1}<0$: similar to (1) using the ideas of the auxiliary claim (3).
        
        \item $\zeta_n<0,\sigma_{n+1}>0$: similar to (3). \qedhere
    \end{enumerate}
\end{proof}

\subsection{Proof of Theorem B}

\begin{proof}[Proof of \Cref{thm:B}]
    We show that
    \[
    \tr(\rho(\mu^{6\sum\limits_{i=1}^{n+1}\sigma_i}\lambda))=\pm 2\in \cO_{\mathfrak K}.
    \]
    Then by \Cref{boundary_slope}, this implies $\mu^{6\sum_{1}^{n+1}\sigma_i}\lambda$ is the boundary slope of the essential surface detected by this action.

    According to \Cref{knotgrouppres}, the meridian is $\mu=a$, and the longitude is $\lambda=w^*wa^{\epsilon}$. By definition, we know $w=b^{-1}w_b^+(\alpha)$, and $w^*b=w_b^-(\alpha)^{-1}$, $w^*=w_b^-(\alpha)^{-1}b^{-1}$.
    Since $K$ is a knot, $\sum_{i=1}^{n+1}\sigma_i$ is odd by \Cref{lem:epsumodd}, so we apply \Cref{computation_of_w}(2). 
    Then following \Cref{computation_of_epsilon}, for $k=\#\{1\le i\le n+1\mid \sigma_i<0\}-\#\{1\le i\le n\mid \zeta_i<0\}$, we know $\epsilon=-4k-6\sum_1^{n+1}\sigma_i+2$, and we may compute that
    \begin{align*}
        \rho \left( \mu^{6\sum\limits_{i=1}^{n+1}\sigma_i} \lambda \right) &=\rho\left(a^{6\sum\limits_{i=1}^{n+1}\sigma_i}w_b^-(\alpha)^{-1}b^{-2}w_b^+(\alpha)a^\epsilon\right)
        \\ &=A^{6\sum\limits_{i=1}^{n+1}\sigma_i}W_b^-(\alpha)^{-1}B^{-2}W_b^+(\alpha)A^\epsilon
        \\ & =\pm A^{6\sum\limits_{i=1}^{n+1}\sigma_i}\cdot A^{2k}BAB\cdot B^{-2}\cdot BABA^{2k}\cdot A^\epsilon
        \\ &=\pm (A^{6\sum\limits_{i=1}^{n+1}\sigma_i}\cdot A^{2k}BA)(ABA^{2k}\cdot A^\epsilon)
        \\ & =\pm (BAB^{2k+6\sum\limits_{i=1}^{n+1}\sigma_i})(B^{2k+\epsilon}AB)
        \\ &=\pm BAB^{4k+6\sum\limits_{i=1}^{n+1}\sigma_i-4k-6\sum\limits_{i=1}^{n+1}\sigma_i+2}AB
        \\ &=\pm (BAB)(BAB)=\pm I.
    \end{align*}

    Therefore, we conclude that
    \[
    \tr \left( \rho \left( \mu^{N}\lambda \right) \right) = \tr \left( \pm \begin{bmatrix}
        1&0\\ 0&1
    \end{bmatrix} \right) = \pm 2,
    \]
    where $N=6\sum\limits_{i=1}^{n+1}\sigma_i$.
    By \Cref{boundary_slope} this means $\mu^{6\sum_{1}^{n+1}\sigma_i}\lambda$ is the boundary slope of the essential surface detected by this action.
\end{proof}